\documentclass{article} % Specifies the document style.
\usepackage[utf8]{inputenc}
\usepackage{amsthm}
\usepackage{amsfonts}
\usepackage{amsmath}
\usepackage{amssymb}
\usepackage[arrow, matrix, curve]{xy} 
\usepackage{mathtools}
\usepackage{mathrsfs}
\usepackage{enumitem}
\usepackage[toc,page]{appendix}
\usepackage[square,sort,comma,numbers]{natbib}

\theoremstyle{plain}
\newtheorem{thm}{Theorem}[section]
\newtheorem{lemma}[thm]{Lemma}
\newtheorem{corollary}[thm]{Corollary}
\newtheorem{proposition}[thm]{Proposition}

\theoremstyle{definition}

\newtheorem{remark}[thm]{Remark}

\newtheorem{example}[thm]{Example}
\usepackage{mathrsfs}

\newcommand{\Cinf}{\mathbb{C}_{\infty}}
\newcommand{\C}{\mathbb{C}}

\newcommand{\D}{\mathbb{D}}
\newcommand{\T}{\mathbb{T}}

\newcommand{\linspan}{{\rm span}}
\newcommand{\om}{\Omega}

\newcommand{\vp}{\varphi}

\begin{document}

\title{Dynamics of the Taylor shift on Bergman spaces}
\author{Jürgen Müller and Maike Thelen}

\maketitle

\begin{abstract} The Taylor (backward) shift on Bergman spaces $A^p(\om)$ for general open sets $\om$ in the extended complex plane shows rich variety concerning its dynamical behaviour. Different aspects are worked out, where in the case $p<2$ a recent result of Bayart and Matheron plays a central role.    
\end{abstract}



\section{Introduction}

Let $\Omega$ be an open subset of the Riemann sphere $\C_\infty$, where $\Cinf$ is equipped with the spherical metric. Moreover, let $H(\om)$ denote the Fr{\'e}chet space of functions holomorphic in $\om$ and vanishing at $\infty$, endowed with the topology of compact convergence. If $0 \in \Omega$, the  Taylor (backward) shift $T: H(\om) \to H(\om)$ is defined by
\[
 (Tf)(z):=\begin{cases}(f(z)-f(0))/z, & z \not =0\\   f'(0), & z=0 \end{cases}.
 \] 
It is easily seen that $T$ is a continuous operator on $H(\Omega)$. Moreover, the $n$-th iterate $T^n$ is given by 
\[
 (T^n f)(z):=\begin{cases}(f-S_{n-1}f)(z)/z^n, & z \not =0\\   a_n, & z=0 \end{cases},
 \]
where
$(S_n f)(z):=\sum_{\nu=0}^n a_\nu z^\nu$
 denotes the $n$-th partial sum of the Taylor expansion of $f$ about $0$. In particular, for  $|z|<{\rm dist}(0, \partial \om)$ we have 
 \[
 ( T^nf)(z)=\sum_{\nu=0}^\infty a_{\nu+n}z^\nu\,, 
  \]
that is, locally at $0$ the Taylor shift acts as backward shift on the Taylor coefficients. 

An important feature of the Taylor shift is that the spectrum is easily determined:  Setting $A^* = 1/(\Cinf \setminus A)$ for $A \subset \C_\infty$, the set $\om^*$ is compact in the complex plane $\C$ if and only if $\om$ is open in $\Cinf$ with $0 \in \om$. For $\alpha \in \C$, we define $\gamma(\alpha) \colon \{\alpha \}^* \to \C$ by
\begin{equation} \label{eigenfunctionsOfTaylor}
 \  \gamma (\alpha)(z)  \coloneqq \frac{1}{1-\alpha z} \quad (z \in \C_{\infty} \setminus \{1/\alpha\}).
\end{equation}
Since $\gamma(\alpha) \in H(\om)$ is an eigenfunction to the eigenvalue $\alpha$ for all $\alpha \in \om^*$, the point spectrum contains $\om^*$. Moreover, the corresponding eigenspace is one-dimensional. On the other hand, a calculation shows that  for $1/\alpha \in \Omega$ the operator $S_{\alpha}: H(\Omega) \to H(\Omega)$ defined by
\begin{equation*} \label{inverseSAlpha}
 (S_{\alpha}g)(z) = \frac{zg(z)-g(1/\alpha)/\alpha}{1-z\alpha} \qquad (z \in \om\setminus \{1/\alpha\})
\end{equation*}
(and appropriately extended at $1/\alpha$) is the continuous inverse to $T - \alpha I$ and hence the spectrum and the point spectrum both equal $\Omega^*$. 

The Taylor shift may also be considered as an operator on Banach spaces of functions holomorphic in $\om$ as e.g. Bergman spaces, that is, subspaces of $H(\om)$ of functions which are $p$-integrable with respect to the two-dimensional Lebesgue measure. In the case of the open unit disc $\om=\D$ there is an elaborated theory about invariant subspaces and cyclic vectors for Hardy- and Bergman spaces (see e.g. \cite{CimaRoss}, cf.\ also \cite{FKMR}).  Since we are interested also -- and in particular -- in the case of open sets $\om$ containing $\infty$ and in order to avoid difficulties according to local integrability at $\infty$, we modify the usual Bergman spaces and consider the  surface measure on the sphere $\Cinf$ instead. We denote the normalized surface measure by $m_2$ and, correspondingly, the normalized arc length measure on the unit circle $\T$  by $m_1$ of briefly $m$.  

For $1 \leq p < \infty$ and $\om \subset \Cinf$ open we define the Bergman space $A^p(\Omega)=A^p(\om,m_2)$ as the set of all functions $f \in H(\Omega)$ which fulfil 
\begin{equation*}
\Vert f \Vert_p :=\left( \int_\Omega \vert f \vert^p \, dm_2 \right)^{1/p} < \infty\, .
\end{equation*}
Then $(A^p(\Omega), \Vert \cdot\Vert_p)$ is a Banach space. If $\om$ is open and bounded in $\C$, the above norm and the classical $p$-norm with respect to Lebesgue measure are equivalent.

In case $0 \in \om$, the Taylor shift turns out to be a continuous operator on $A^p(\om)$. 
For $\alpha \in (\Omega^*)^\circ$ the functions $\gamma(\alpha)$ belong to $A^p(\Omega)$ for all $p$ and it is clear that $\gamma(\alpha)$ is an eigenfunction to the eigenvalue $\alpha$. Again, for $1/\alpha \in \om$,  the operator $S_\alpha$ from above, now defined on $A^p(\om)$, turns out to be the continuous inverse to $T - \alpha I$. 
Moreover, in the case $p<2$ the functions $\gamma (\alpha)$ belong to $A^p(\Omega)$ also for $\alpha \in \partial \Omega^*$, which yields that in this case the point spectrum  equals $\Omega^*$. Thus, we obtain:
\begin{enumerate}
\item $(\Omega^*)^\circ \subset \sigma_0(T)$ and $\overline{(\om^*)^\circ}  \subset \sigma(T) \subset \Omega^*$ for all $p \ge 1$.
\item $\sigma_0(T) = \sigma(T) = \Omega^*$ for $1 \leq p <2$.
\end{enumerate}   
 This gives high flexibility in prescribing spectra. In particular,  each compact plane set $K$ appears as spectrum and point spectrum of $T$ on $A^p(K^*)$ for $1 \le p <2$.  For $p \ge 2$ the situation is more delicate.  In general $\gamma(\alpha)$ does not belong to $A^p(\om)$ for $\alpha \in \partial \om^*$. If, however, $\om$ is "sufficiently small" near a boundary point $1/\alpha$ of $\om$, it may happen that $\gamma(\alpha)$ does belong to $A^p(\om)$. A simple example is the crescent-shaped region $\om=\D \setminus \{z:|z-1/2|\le 1/2\}$, where $\gamma(1) \in A^2(\om)$. This opens the possibility to choose $\om$ in such a way that eigenvalues are placed at certain points.  \\

In \cite{BMM}, \cite{BM} and  \cite{TaylorShiftArticle}, the behaviour of the Taylor shift with respect to topological dynamics was studied. We recall that an operator $T$  on a separable Fr\'echet space $X$ is called  \textit{hypercyclic} if $T$ has a dense orbit. This is equivalent to $T$ being \textit{topologically transitive}, that is, for any two nonempty open sets $U, V  \subset X$ the images $T^n(U)$ meet $V$ infinitely often. Moreover, $T$ is \textit{topologically mixing} if $T^n(U)$ meets $V$ for all sufficiently large $n$. Concerning these and further notions from topological (linear) dynamics we refer the reader to \cite{BayartMatheron} and \cite{LinChaos}. 

The main result from \cite{BMM} states that the following are equivalent: 
\begin{itemize}
\item $T$  is topologically mixing
\item $T$ is hypercyclic
\item Each component of $\om^*$ meets the unit circle $\T$. 
\end{itemize}
The situation changes drastically if we consider Bergman spaces.  If $T$ is  hypercyclic on $A^p(\om)$, for some  $p<2$, then $\om^*$ has to be perfect. In \cite{BM} it is shown that $T$ is mixing on $A^p(\om)$ if $\om \supset \D$ is a Jordan domain such that $\om^\ast\cap \T$ contains an arc.  

In Section 2 we study the Taylor shift operator for its metric dynamical properties. For $H(\om)$ and in the case of $A^p(\om)$ with $p<2$, the sufficient supply of eigenvectors $ \gamma(\alpha)$ allows the application of a recent deep result of Bayart and Matheron (Theorem 1.1 from \cite{BayMatArt}) which in many respects finishes a line of investigations concerning relations between the existence of unimodular eigenvectors and the dynamics of a linear operator. 

This is no longer possible for $p\ge 2$. In Section 3  we show that the Taylor shift is topologically mixing on $A^p(\om)$ for arbitrary $p$ if each component of $\om^*$ is sufficiently large near the unit circle $\T$.

\section{Metric dynamics of $T$} 

In this section, we investigate the Taylor shift on $H(\om)$ and $A^p(\om)$ for $p<2$ with respect to its metric dynamical behaviour. We recall that a measure-preserving transformation $T$ on a probability space $(X, \Sigma, \mu)$ is called \textit{weakly mixing (with respect to $\mu$)} if for any $A,B \in \Sigma$
\begin{equation*}
\lim\limits_{N \to \infty} \frac{1}{N}\sum\limits_{n=0}^{N-1}\vert \mu(A\cap T^{-n}(B))-\mu(A)\mu(B)\vert = 0
\end{equation*} 
and it is called \textit{strongly mixing (with respect to $\mu$)} if for any $A,B \in \Sigma$
\begin{equation*}
\mu (A \cap T^{-n}(B)) \to \mu(A)\mu(B) \quad (n \to \infty).
\end{equation*}
For further notions from ergodic theory we refer the reader e.g. to \cite{Walters}. 
Consider now $X$ to be a complex separable Fr{\'e}chet space. Each operator $T$ which is weakly mixing with respect to some measure of full support is frequently hypercyclic (see e.g. \cite[Corollary 5.5]{BayartMatheron}) and then also hypercyclic. Moreover, strong mixing with respect to some measure of full support implies topological mixing.
 
The operator $T$ is called weakly (resp. strongly) mixing \textit{in the Gaussian sense} if it is weakly (resp. strongly) mixing with respect to some Gaussian probability measure $\mu$ having full support. The definition of Gaussian probability measures and related results can be found in \cite{BayartMatheron} and \cite{BayMatArt}. For the notion of the cotype of a Banach space we refer to  \cite{TopicsInBanachSpaceTheory}. 
 
%If $T$ a continuous linear operator on $X$
%then $T$ has \textit{perfectly spanning $\mathbb{T}$-eigenvectors} if for all countable sets $D \subset \mathbb{T}$  the linear span of
%$\bigcup_{\lambda \in \mathbb{T} \setminus D} \textnormal{ker}(T -\lambda I)$
% is dense in $X$. 
%Furthermore,  $T$ has \textit{$\mathcal{U}_0$-perfectly spanning $\mathbb{T}$-eigenvectors} if the same condition holds for all sets $D\in \mathcal{U}_0$.

Let now $T$ be the Taylor shift on $H(\om)$ or $A^p(\om)$, where $1\le p<2$. In order to treat both cases simultaneously, we write $A^0(\om):=H(\om)$. Then, for $D \subset \T$, 
\begin{equation*}
\textnormal{span}\bigcup\limits_{\alpha \in \mathbb{T} \setminus D} \textnormal{ker}(T- \alpha I)=  {\rm span}\big(\gamma(\om^\ast \cap \T \setminus D)\big).
\end{equation*}
For $\Lambda \subset \T$ we say that $\gamma(\Lambda)$ is \textit{perfectly spanning} in $A^p(\om)$  if the span of  $\gamma(\Lambda \setminus D)$ is dense in $A^p(\om)$ for all countable $D \subset \T$. Similarly, we say that $\gamma(\Lambda)$ is \textit{$\mathcal{U}_0$-perfectly spanning} if this holds for all $D \in \mathcal{U}_0$, where $\mathcal{U}_0$ denotes class of sets of extended uniqueness (see e.g. \cite[p. 76]{DescriptiveSetTheory}). We recall that all sets of extended uniqueness have vanishing arc length measure. 

Since for $1\le p\le 2$ the Bergman space $A^p(\om)$ as closed subspace of $L^2(\om,m_2)$ is of cotype 2, we obtain as  an immediate consquence of the Bayart-Matheoren theorem  mentioned in the introduction (Theorem 1.1 from \cite{BayMatArt})

\begin{thm} \label{perfectlySpanningCrit}
Let $0 \in \Omega \subset \Cinf$ be an open set and let $T$ be the Taylor shift on $A^p(\om)$, where $p\in \{0\} \cup[1,2)$.
\begin{enumerate}
\item If $\gamma(\om^\ast \cap \T)$ is perfectly spanning in $A^p(\om)$ then $T$ is weakly mixing in the Gaussian sense.
\item If  $ \gamma(\om^\ast \cap \T)$ is $\mathcal{U}_0$-perfectly spanning  in $A^p(\om)$ then $T$ is strongly mixing in the Gaussian sense.
\end{enumerate}
If $p\in [1,2)$ then in both cases the converse implication is true.  
\end{thm} 
We say that a point $z \in \C$ is a \textit{perfect limit point} of $A  \subset \C$ if $U \cap A$ is uncountable for each neighbourhood $U$ of $z$, that is, if $z$ is a limit point of $A \cap U \setminus D$ for each countable set $D$.  Similarly, we say that $z$ is a \textit{$\mathcal{U}_0$-perfect limit point} if $z$ is a limit point of $A \cap U  \setminus D$ for each neighbourhood $U$ of $z$ and each $D \in \mathcal{U}_0$.  If $A \subset \T$ has locally positive arc length measure at $z$ then $z$ is a $\mathcal{U}_0$-perfect limit point.   
Applying an appropriate version of Runge's theorem which can be found e.g.\ in \cite[Theorem 10.2]{LueckingRubel} we obtain from Theorem \ref{perfectlySpanningCrit}:
\begin{corollary}\label{TaylorH}
Let $0 \in \Omega \subset \Cinf$ be an open set and let $T$ be the Taylor shift on $H(\Omega)$.
\begin{enumerate}
\item If each component of $\om^*$ contains a perfect limit point of $\Omega^* \cap \mathbb{T}$, then $T$ is weakly mixing in the Gaussian sense. 
\item If each component of $\om^*$ contains a $\mathcal{U}_0$-perfect limit point of $\Omega^* \cap \mathbb{T}$, then $T$ is strongly mixing in the Gaussian sense.
\end{enumerate}
\end{corollary}
\begin{remark} 
By separating singularities it is easily seen from \cite[Corollary 1]{TaylorShiftArticle} that $\om^\ast$ necessarily has to be  perfect if the Taylor shift  on $H(\om)$ is weakly mixing (or merely frequently hypercyclic). Note that $\om^\ast \cap \T$ not necessarily has to be perfect: If $B$ is some closed arc on $\T$ symmetric to the real axis  and  
\[
\om=\C \setminus \big(B \cup (-\infty,-1]\cup [1, \infty)\big)
\]
 then $\om^\ast=B \cup [-1,1]$ satisfies the assumption of Corollary \ref{TaylorH}.2, hence $T$ is strongly mixing on $H(\om)$.
\end{remark}
We turn to Bergman spaces. Theorem \ref{perfectlySpanningCrit} shows that the question whether $T$ is (strongly or weakly) mixing completely reduces to a question about  mean approximation by rational functions with simple poles in appropriate subsets of $\T$.  The following result on separation of singularities implies that the general question may be reduced to special cases. 
\begin{proposition} \label{thmDirectSumBergmanSpace}
Let $p \ge 1$ and let  $\Omega_1, \om_2 \subset \Cinf$ be open sets in $\C_\infty$ with $\om_1 \cup \om_2=\Cinf$. Then
$A^p(\Omega_1 \cap \om_2) = A^p(\Omega_1) \oplus A^p(\om_2)$. 
\end{proposition}
\begin{proof}
It is known that, by separation of singularities of holomorphic functions, 
\[
H(\om_1 \cap \om_2) = H(\Omega_1) \oplus H(\om_2)
\]
as topological direct sum. Since convergence in $A^p(\om)$ implies convergence in $H(\om)$ (see e.g.\ \cite[Chapter 1, Theorem 1]{DurenSchuster}), it suffices to show that for $f \in A^p(\om_1 \cap \om_2)$ decomposed as  $f=f_1  + f_2 \in H(\om_1)\oplus H(\om_2)$  we have $f_j \in A^p(\om_j)$ for $j=1,2$. 

Let $f \in A^p(\Omega_1 \cap \om_2)$ and $f=f_1 + f_2$ with $f_j \in H(\om_j)$. Since the boundary of $\om_1 \cap \om_2$ is the union of the (compact) boundaries $\partial \om_j\subset\om_{3-j}$, for $j=1,2$, we can find compact disjoint neighbourhoods $U_j\subset \om_{3-j}$ of $\partial \om_j$. Then
\begin{equation*}
\int_{\om_j} \vert f_j\vert^p dm_2 = \int_{\om_j \setminus U_j} \vert f_j \vert^p dm_2 + \int_{\om_j \cap U_j} \vert f-f_{3-j} \vert^p dm_2 < \infty.
\end{equation*}
This yields $f_j \in A^p(\om_j)$ for $j=1,2$. 
\end{proof}
An immediate consequence is the fact that hypercyclicity of $T$ on $A^p(\om)$, for some $1\le p<2$, implies that $\om^\ast$ is perfect: Suppose that $\zeta$ is an isolated point of $\om^*$. Then  we have 
\[
A^p(\om)= A^p(\om \cup  \{1/\zeta\})\oplus  A^p(\Cinf\setminus \{1/\zeta\}).
\]
By \cite[Proposition 2.25]{LinChaos} it follows that $T$ is also hypercyclic on $A^p(\Cinf\setminus \{1/\zeta\})$. Since $A^p(\Cinf \setminus \{1/\zeta\})$ reduces to the span of $\gamma(\zeta)$ and is thus one-dimensional we get a contradiction.

In order to be able to reduce the case of open sets $\om$ containing $\infty$ to the case of bounded open sets in $\C$ we recall 
\begin{proposition}\label{approximationForUnboundedSets}
Let $X$ a Fr{\'e}chet space and let $L$ be complemented in $X$. If $X=L \oplus M$ and if $A \subset L$ and $B \subset M$  with $\linspan(A+B)$ dense in $X$ then $\linspan(A)$ is dense in $L$. 
\end{proposition}
\begin{proof}
Let $a \in L$. Then a sequence $(x_n)$ in $\linspan(A+B)$ exists with $x_n \to a$ in $X$ as $n$ tends to $\infty$. We write  $x_n = a_n + b_n$ with $a_n \in {\rm span}\, A$ and $b_n \in {\rm span}\, B$. Since $a$ belongs to $L$ and since the projection of $X$ to $L$ along $M$ is continuous (see e.g.\ \cite[Theorem 5.16]{RudinFunctionalAnalysis}), the sequence $(a_n)$ converges to $a$.
\end{proof}
\begin{remark}\label{decomp}
Let $\Omega$ be open with $\infty \in \Omega$ and $\rho> \max_{z \in \C_\infty \setminus \om}|z|$. If we put $\om_\rho:=\om \cap \rho\D$ then 
\[
 A^p(\om_\rho)=A^p(\om)\oplus A^p(\rho\D) 
\] 
and $\gamma(\rho^{-1} \D) \subset A^p(\rho\D)$ for all $p\ge 1$. 
If $B \subset A^p(\om)$ is so that  $B+\gamma(\rho^{-1} \D)$  densely spans $A^p(\om_\rho)$ then  $B$ has dense span in $A^p(\om)$  by Proposition \ref{thmDirectSumBergmanSpace}. 
\end{remark}

For $K \subset \C$ compact, a set $\Lambda \subset K$ is called a \textit{$K$-uniqueness set} if every continuous function on $K$ which is holomorphic in the interior of $K$ and vanishes on $\Lambda$ vanishes identically. Obviously, if $K$ is nowhere dense then $\Lambda$ is a $K$-uniqueness set if and only if $\Lambda$ is dense in $K$. More generally, it is easily seen that $\Lambda \subset K$ is a $K$-uniqueness set if and only if $K\setminus \overline{K^\circ} \subset \overline{\Lambda}$ and for every component $C$ of $K^\circ$ the set $\Lambda \cap \overline{C}$ is a uniqueness set for $\overline{C}$.  

With that notion, we have the following result on rational approximation. For the case $p=1$ and Lebesgue measure instead of surface measure the result is due to Bers (\cite{BersL1Approximation}). 

\begin{thm}\label{uniquenessSetApproximation}\label{TaylorOnA1}
Let $1\leq p<2$ and $0 \in \Omega \subset \Cinf$ be an open set which is either bounded in $\C$ or contains $\infty$. Moreover,  suppose $\Lambda$ to be a subset of $\om^*$.  
\begin{enumerate}
\item If $\Lambda$ is a $\om^*$-uniqueness set then the span of $\gamma( \Lambda)$ is dense in $A^p(\Omega)$.
\item If  $m_2(\om^*)=0$ then the span of $\gamma(\Lambda)$ is dense in $A^p(\Omega)$ if and only if $\Lambda$ is dense in $\om^*$. 
\end{enumerate}
\end{thm}
\begin{proof}
1. We first assume that $\Omega$ is bounded in $\C$. Let $\ell \in A^p(\Omega)'$ with $\ell (\gamma (\alpha)) = 0$ for all $\alpha \in \Lambda$. Since $A^p(\Omega)$ is a subspace of $L^p(\Omega)$ the Hahn-Banach theorem yields that $\ell$ can be extended to a continuous linear functional on $L^p(\Omega)$. Thus, there exists a function $g \in L^q(\Omega)$, where $q$ is the conjugated exponent, such that 
\begin{equation*}
\ell(f)  =\int_\Omega f\overline{g}\, d m_2 
\end{equation*}
for all $f\in A^p(\Omega)$. For the measure  $1_\Omega g dm_2 \in M(\overline{\Omega})$ the Cauchy transform 
\begin{equation*}
(Vg)(\alpha):=\int_\Omega \frac{\overline{g}(\zeta)}{1-\zeta \alpha} dm_2(\zeta) = \frac{1}{\alpha}\int_\Omega \frac{\overline{g}(\zeta)}{1/\alpha-\zeta} dm_2(\zeta)
\end{equation*}
of $1_\om g dm_2$ is holomorphic in the interior of $\Omega^*$ and continuous in $\C$ as the convolution of $w \mapsto 1/w \in A_p(\C_\infty \setminus\{0\})$ and the function $1_\om g \in L^q(\C)$. Since 
\[
(Vg) (\alpha) = \ell (\gamma (\alpha)) = 0
\]
 for all $\alpha \in \Lambda$ and  since $\Lambda$ is a $\Omega^*$-uniqueness set  we have that $Vg \vert_{\Omega^*}= 0$ and thus 
$\ell(\gamma (\alpha)) = 0 $ for all $\alpha \in \Omega^*$.
So $\ell$ vanishes on the set of rational functions with simple poles in $\C \setminus \Omega$. According to (the proof of) \cite[Theorem 1]{Hedberg}, the set of these functions is dense in $A^p(\om)$. This yields that $\ell=0$ and then the Hahn-Banach theorem implies the assertion. 

Now, let $\Omega$ be open with $\infty \in \Omega$ and $\om_\rho$ as in Remark \ref{decomp}. Then $\om_\rho^\ast=\om^\ast \cup \rho^{-1}\overline{\D}$. 
Since $\Lambda \cup \rho^{-1}\D$ is a $\om_\rho^\ast$-uniqueness set,  by the previous considerations we have that the span of $\gamma(\Lambda \cup \rho^{-1}\D)$ is dense in $ A^p(\Omega_\rho)$. By  Remark \ref{decomp} the span of $\gamma(\Lambda)$ is dense in $A^p(\Omega)$.

2.  It is easily seen that in case $m_2(\om^\ast)=0$ the denseness of $\Lambda$ in $\om^*$ is necessary for $\gamma(\Lambda)$ to be densely spanning in $A^p(\om)$. Conversely, since $\om^*$ is nowhere dense, denseness of $\Lambda$ in $\om^*$ implies $\om^*$-uniqueness. 
\end{proof}
If $\om^*$ has interior points then $\om^\ast$-uniqueness of $\Lambda$ is in general not necessary for $\gamma(\Lambda)$ to be (even perfectly) spanning in $A^p(\om)$:     
 \begin{example} 
  Let $0<\delta<1$ and $E_\delta:=(1+iC_\delta)$, where $C_\delta$ is the convex hull of the closed curve bounded by  $\{t+i\varphi(t):-\delta \le t\le\delta\}$ with
\[
\varphi(t):=e^{-1/|t|}+1-\sqrt{1-t^2} \quad (-\delta\le t\le \delta)
\]
(where $e^{-\infty}:=0$) and the horizontal line $\{t+i\varphi(\delta) : -\delta\le t\le \delta\}$.  Since each dense subset of $\T$ is a $\D^\ast$-uniqueness set, $\gamma(\T)$ is $\mathcal{U}_0$-perfectly spanning in $A^p(\D)$ for $p<2$. For the crescent-shaped domain $\om:=\D \setminus E_\delta$, however,  $\T$ is no $\om^*$-uniquensess set. On the other hand, the domain $\om$ is so "sharp" near the point $1$ that the polynomials form a dense subspace of $A^2(\om)$ (see Theorem 12.1 in \cite{Mer}; cf. also \cite[p. 29]{Gai}) and thus of $A^p(\om)$ for $p < 2$.  In particular, $A^p(\D)$ is dense in $A^p(\om)$. But then $\gamma(\T)$ is also  $\mathcal{U}_0$-perfectly spanning in $A^p(\om)$ for $p<2$  and, according to Theorem \ref{perfectlySpanningCrit}, the Taylor shift on $A^p(\om)$ is strongly mixing in the Gaussian sense. 
\end{example}

From the second part of Theorem \ref{uniquenessSetApproximation} we obtain a quite complete characterization of the metric dynamics of $T$ for the case of open sets $\om$ with $m_2(\om^*)=0$.
We recall that for any perfect set $A \subset \T$ each point in $A$ is a perfect limit point.  A closed set $A \subset \T$ is said to be \textit{$\mathcal{U}_0$-perfect} if $U \cap A \not \in \mathcal{U}_0$ for all open sets $U$ that meet $A$. In particular, closed sets $A \subset \T$ which have locally positive arc length measure are $\mathcal{U}_0$-perfect.  For any  $\mathcal{U}_0$-perfect set $A \subset \T$  each point in $A$ is a $\mathcal{U}_0$-perfect limit point. 
\begin{thm} \label{weaklyMixingEquiv}
Let $\om\subset \Cinf$ be open with $0, \infty \in \om$ and  $m_2(\om^*)=0$. Furthermore, let  $1 \leq p <2$ and $T$ be the Taylor shift on $A^p(\om)$. 
\begin{enumerate}
\item $T$ is weakly mixing in the Gaussian sense if and only if $\om^*$ is a  perfect subset of  $\T$,
\item $T$ is strongly mixing in the Gaussian sense if and only if $\om^*$ is a  $\mathcal{U}_0$-perfect subset of  $\T$.
\end{enumerate}
\end{thm}
\begin{proof}
If $\om^*\subset \T$ is  perfect, then  $\om^*\setminus D$ is dense in $\om^*$ for all countable sets $D$.  Theorem \ref{perfectlySpanningCrit} and  Theorem \ref{TaylorOnA1} show that $T$ is weakly mixing in the Gaussian sense.
In the same way,  Theorem \ref{perfectlySpanningCrit} and  Theorem \ref{TaylorOnA1} show that $T$ is strongly mixing in the Gaussian sense if $\om^* \subset \T$ is $\mathcal{U}_0$-perfect. 

On the other hand, as noted above, already  hypercylicity of $T$ requires perfectness of $\om^*$. Since $m_2(\om^*)=0$,  Theorem \ref{perfectlySpanningCrit} and Theorem \ref{TaylorOnA1} show that the set  $\om^*$ has to be a subset of $\T$.  
 Moreover,
 \[
\om^*\ni \alpha \to \gamma(\alpha) \in A^p(\om)
\]
 defines a continuous eigenvector field for $T$. The same arguments as in Example 2 of \cite{BayMatArt} show that $\mathcal{U}_0$-perfectness of $\om^*\cap \T$ is necessary for $T$ to be strongly mixing on $A^p(\om)$ for any $1\le p<2$.  
 \end{proof}
\begin{example}\label{arc}

Theorem \ref{weaklyMixingEquiv} implies that for each set $B \subset \T$ which has locally positive arc length measure (as e.g.\ a nontrivial arc) the Taylor shift $T$ on $A^p(\Cinf \setminus B)$ is strongly mixing in the Gaussian sense for $p<2$. If $B$ is perfect but not $\mathcal{U}_0$-perfect then $T$ is weakly mixing but not strongly mixing in the Gaussian sense.   
\end{example}

For the case that $\om^*$ has interior points we can show
\begin{thm}
Let $0\in \om \subset \C_\infty$ be an open set which is either bounded in $\C$ or contains $\infty$. If each component $K$ of $\Omega^*$ is the closure of a simply connected domain $G$ such that the harmonic measure $\omega(\cdot,K \cap \T,G)$ is positive or $G$ meets $\T$ then the  Taylor shift on $A^p(\Omega)$ is strongly mixing in the Gaussian sense for all $p<2$.   
 \end{thm}
 
 \begin{proof}
From the two-constant-theorem (see e.g. \cite{Ransford}) it follows that for each domain $G$ with non polar boundary sets $A \subset \partial G$ of positive harmonic measure $\omega(\cdot,A,G)$ are uniqueness sets for $G$.  If $K$ is a component of $\Omega^*$ then, according to our assumptions, the local F.\ and M.\ Riesz theorem (see \cite[p. 415]{Garnett}) shows that  $m(K \cap \T)$ is positive. Since each $\mathcal{U}_0$-set $D \subset \om^\ast \cap \T$ has vanishing arc length measure and, again by the local F.\ and M.\ Riesz theorem, also vanishing harmonic measure, $(\om^*\cap \T) \setminus D$ is a $\Omega^*$-uniqueness set for all $D \in \mathcal{U}_0$. Hence, according to  Theorem \ref{perfectlySpanningCrit} and  Theorem \ref{TaylorOnA1} the  Taylor shift on $A^p(\Omega)$ is strongly mixing in the Gaussian sense for all $p<2$.  
\end{proof}
\begin{remark}\label{suff}
Let $\om$ with $0 \in \om$ be the exterior of a rectifiable Jordan curve $\Gamma$. Then the interior $G=(\om^\ast)^\circ$ of $1/\Gamma$ is a Jordan domain with rectifiable boundary  and,  according to the (global) F.\ and M.\ Riesz theorem (see e.g.\ \cite[p. 202]{Garnett}), the harmonic measure of a set $A \subset \partial G$ is positive if and only if the linear measure is positive. For $A \subset \T$ this in turn is equivalent to $A$ having positive arc length measure. Hence, if $m(\om^\ast \cap \T)>0$, then $T$ is strongly mixing in the Gaussian sense for all $p<2$.

Let $\lambda_2$ denote the two-dimensional Lebesgue measure and let 
\[
D_q(G):=\{ h \in H(G): \int_G |h'|^q\, d \lambda_2<\infty\}
\] 
be the Dirichlet space of order $q$ with respect to $G$. In a similar way as in the proof of Theorem 1 in \cite{NY}, by applying Theorem 3, Chapter II, Section 4,  from \cite{Ste}, it can be shown that for Cauchy transforms $Vg$ of functions $g \in L^q(\om)$ as considered in the proof of Theorem   \ref{uniquenessSetApproximation} the restrictions $Vg|_G$ belong to $D_q(G)$
and thus in particular to $D_2(G)$. It is known that for the Dirichlet space $D_2(\D)$ perfect uniqueness sets of vanishing arc length measure exist   
(see \cite[Corollary 4.3.4]{FKMR}). By conformal invariance (cf. \cite[Therorem 1.4.1]{FKMR}), the rectifiable Jordan curve $\Gamma$ can be chosen in such a way that $m(\om^\ast \cap \T)=0$ and that $T$ is weakly mixing in the Gaussian sense for all $p<2$.
\end{remark}

\begin{example} \label{CantorSetPosLebesgueMeasure}
Let $C \subset \T$ be a closed set and consider $\Gamma$ to be a rectifiable Jordan curve in $\C \setminus \D$ with $\Gamma \cap \mathbb{T} = C$ and  so that the exterior $\om$ of $\Gamma$ contains $0$. If $C$ has positive arc length measure then Remark \ref{suff} shows that the Taylor shift $T$  on $A^p(\Omega)$ is strongly mixing in the Gaussian sense for $p< 2$. 
Note that $C$ may be chosen to be a totally disconnected and that $C$ may have isolated points. Moreover, according to the proof of \cite[Corollary 4.3.4]{FKMR}, for an appropriate countable union $C$ of the circular Cantor middle-third set the Taylor shift  weakly mixing in the Gaussian sense for $p<2$.   
\end{example}

\section{Topological dynamics of $T$}

If $\om$ is an open set such that no point of $\T$ is an interior point of $\om$ and if $p \ge2$, the Taylor shift $T$ on $A^p(\om)$ may have no unimodular eigenvalues. This is e.g.\ the case for $\om=\D$. Since $A^2(\om)$ is of cotype 2, Theorem \ref{perfectlySpanningCrit} shows that weak mixing in the Gaussian sense is excluded. We recall that an operator $T$ on a Fr\'echet space $X$ is called \textit{frequently hypercyclic} if the orbit of some point $x$ meets each nonempty open set with positive lower density. Each operator that is weakly mixing with respect to some measure of full support is frequently hypercyclic. 

The space $A^2(\D)$ is isometrically isomorphic to the weighted sequence space $\ell^2(1/(n+1))$ and the Taylor shift is conjugated to the backward shift on $\ell^2(1/(n+1))$ (see \cite[Example 4.4.(b)]{LinChaos}). As a consequence, the Taylor shift is topologically mixing but not frequently hypercyclic on $A^2(\D)$ (\cite[Example 9.18]{LinChaos}). It turns out that a similar result holds for the Taylor shift on more general domains $\om$ and for arbitrary $p$. 

\begin{thm} \label{mixingTaylorShift}
Let $1\le p <\infty$ and $0\in \Omega\subset \Cinf$ be a domain which is either bounded in $\C$ or contains $\infty$. If each component $K$ of $\Omega^*$ is the closure of a simply connected domain containing a rectifiable Jordan curve $\Gamma$ such that the linear measure of $\Gamma \cap \T$ is positive, then the Taylor shift on $A^p(\Omega)$ is topologically mixing.
If, in addition, $\D\subset \om$ then $T$ is not frequently hypercyclic for any $p \ge 2$.
\end{thm}

\begin{remark}\label{localFandM}
Theorem \ref{mixingTaylorShift} extends a corresponding result in \cite{BM}, where $\om$ is a Caratheodory domain that contains a (nontrivial) subarc of $\T$. It shows, in particular, that in the situation of Example \ref{CantorSetPosLebesgueMeasure} the Taylor shift on $A^p(\om)$ is topologically mixing for all $p\ge 1$ and not frequently hypercyclic for any $p \ge 2$.

According to Example \ref{arc}, for each (nontrivial) arc $B \subset \T$ the Taylor shift on $A^p(\Cinf \setminus B)$ is topologically mixing for $p<2$. We do not know if this is still the case for $p\ge 2$. 
\end{remark}
The remaining part of the section is devoted to the proof of Theorem \ref{mixingTaylorShift}. Our aim is to apply a version of Kitai's Criterion (see \cite[Remark 3.13]{LinChaos}). 

%Kitai's Theorem (see e.g.\ \cite[Theorem 5.6]{LinChaos}) shows that an operator on a Banach space can only be hypercyclic if every component of the spectrum intersects the unit circle. Therefore, we can assume that $\Omega$ is a domain such that $\partial \Omega^*$ meets $\mathbb{T}$.

Let $E\subset  \C$ be compact and let $M(E)$ denote the set of complex measures on the Borel sets of $\C$ with support in $E$. It turns out that the Cauchy transforms of measures $\mu \in M(\Omega^*)$ are of particular interest for analysing the Taylor shift on $A^p(\Omega)$. For $\mu \in M(E)$ the Cauchy transform $C\mu \in H(E^\ast)$ of $\mu$ is defined (in terms of vector valued integration) by  
\[
C\mu:=\int\gamma(\zeta) \, d\overline{\mu}(\zeta)= \int \frac{1}{1-\zeta \, \cdot} \, d\overline{\mu}(\zeta)
\]
 We write $|\mu |$ for the total variation of the measure $\mu$ and set 
\begin{equation*}
\mathcal{M}_p(\om) = \{ \mu \in M(\om^\ast): \int \vert \gamma (\zeta) \vert\, d \vert \mu \vert (\zeta)\in L^p(\Omega) \}
\end{equation*}
as well as
\begin{equation*}
\mathcal{C}_p(\om) = \{C\mu \colon \mu \in \mathcal{M}_p(\om)\}.
\end{equation*}
For $f \in \mathcal{C}_p(\om)$ we denote by $\mathcal{C}^{-1}(f) = \{ \mu \in \mathcal{M}_p(\om): C\mu = f\}$ the set of representing measures for $f$. Note that for $\mathcal{C}_p(\om) \subset  A^p(\Omega)$ since Cauchy transforms of measures $\mu \in \mathcal{M}_p(\om)$ are holomorphic in $\Omega$ and 
\begin{equation*}
|C\mu| \leq \int \vert \gamma (\zeta) \vert\, d \vert \mu \vert (\zeta)\in L^p(\Omega).
\end{equation*}
\begin{lemma} \label{iteratesT^nr_nuOnA^p}
Let $1 \leq p < \infty$ and $0\in \Omega \subset \C_\infty$ be an open set. Then $T(\mathcal{C}_p(\Omega))\subset \mathcal{C}_p(\Omega)$ and for $R \colon \mathcal{M}_p(\om) \to \mathcal{M}_p(\om)$, defined by $d(R\mu )(\zeta)= \overline{\zeta}\,d\mu(\zeta)$, the diagram 
\begin{align*}
\begin{xy}
  \xymatrix{
      \mathcal{M}_p(\om) \ar[r]^R \ar[d]_C    &   \mathcal{M}_p(\om) \ar[d]^C  \\
      \mathcal{C}_p(\om) \ar[r]_T             &   \mathcal{C}_p(\om)   
  }
\end{xy}
\end{align*}
commutes.
\end{lemma}
\begin{proof}
Let $1 \leq p < \infty$ and $\mu \in \mathcal{M}_p(\om)$. We first show that $R$ is a self map. For $c:=\max_{z \in \partial\Omega^*} \vert z \vert$ we obtain
\begin{equation*}
\int \vert \gamma(\zeta) \vert d \vert R \mu \vert (\zeta) = \int \vert \gamma(\zeta) \zeta \vert d \vert \mu \vert (\zeta) \leq c \int \vert \gamma(\zeta) \vert d \vert \mu \vert (\zeta).
\end{equation*}
It follows that $R\mu \in \mathcal{M}_p(\om)$. Now, let $f \in \mathcal{C}_p(\om)$ with $\mu \in \mathcal{C}^{-1}(f)$. Since we can interchange integration and $T$ (cf. \cite[Exercise 3.24]{RudinFunctionalAnalysis}), we obtain
\begin{equation*}
Tf = \int T\gamma(\zeta) d\overline{\mu} (\zeta) = \int \zeta\gamma(\zeta) d\overline{\mu} (\zeta) = \int \frac{\zeta}{1-\zeta \cdot} d\overline{\mu} (\zeta),
\end{equation*}
i.e. $Tf = CR\mu$. Since $R$ is a self map on $\mathcal{M}_p(\om)$, it follows that $Tf \in \mathcal{C}_p(\partial\Omega^*)$.
\end{proof}
Inductively, from Lemma \ref{iteratesT^nr_nuOnA^p} we obtain
\begin{equation}\label{T_n}
T^n f = \int\zeta^n \gamma(\zeta) d \overline{\mu} (\zeta)
\end{equation}
 for $f \in \mathcal{C}_p(\om)$, $\mu \in \mathcal{C}^{-1}(f)$ and $n \in \mathbb{N}$.
In view of Kitai's criterion, our aim is to find measures $\mu$ such that $T^n(C\mu)$ converges to 0 in $A^p(\Omega)$. We shall see that this is the case if $\mu \in \mathcal{M}_p(\om)$ is supported on $\Omega^* \cap \mathbb{T}$ and a Rajchman measure. We recall that a Borel measure $\nu$ supported on $\T$ is called a Rajchman measure if the Fourier-Stieltjes coefficients 
$\hat{\nu}(k)=\int\zeta^k \, d\nu(\zeta)$ tend to $0$ as  $k$ tends to $ \pm \infty$ (see e.g. \cite{DescriptiveSetTheory}). 

Again according to Kitai's criterion, we also need a kind of right inverse of $T$: If $\mu \in  \mathcal{M}_p(\om)$ is a measure with support in $\T$ we define 
\begin{equation} \label{definitionS_n}
S_n \mu: = \int \frac{\gamma(\zeta)}{\zeta^n} d\overline{\mu} (\zeta)= \int \frac{d\overline{\mu} (\zeta)}{\zeta^n(1-\zeta \cdot)}
\end{equation}
for all $n \in \mathbb{N}$. As in the proof of Lemma \ref{iteratesT^nr_nuOnA^p} it is seen that $S_n \mu \in \mathcal{C}_p(\Omega)$.

\begin{lemma} \label{iteratesConvergeTo0}
Let $1\leq p < \infty$ and $0\in \Omega\subset\C_\infty$ be an open set which is either bounded in $\C$ or contains $\infty$. Furthermore, let $T$ be the Taylor shift operator on $A^p(\Omega)$. If $f \in \mathcal{C}_p(\om)$ such that $f$ is represented by a Rajchman measure $\mu_f \in \mathcal{C}^{-1}(f)$ supported on $ \Omega^* \cap \mathbb{T}$ then 
$T^n f \to 0$ and $S_n \mu_f \to 0$
in $A^p(\Omega)$ as $n \to \infty$.
\end{lemma}
\begin{proof}
Let $f \in \mathcal{C}_p(\om)$ and $\mu_f \in \mathcal{C}^{-1}(f)$ such that $\mu_f$ is a Rajchman measure supported on $B:=\Omega^* \cap \mathbb{T}$. We fix $z \in \om$. By \eqref{T_n} we have 
\begin{equation*}
T^nf(z) = \int\limits_{\partial \Omega^*} \frac{\zeta^n}{1-\zeta z}d\overline{\mu_f} (\zeta) = \int\limits_B \frac{\zeta^n}{1-\zeta z}d\overline{\mu_f} (\zeta)
\end{equation*}
for all $n \in \mathbb{N}$. Because $\mu_f \in \mathcal{M}_p(\om)$ is supported on $B \subset \mathbb{T}$, the function $\gamma(z)$ belongs to $L^1(\mathbb{T}, \vert \mu_f \vert)$. Since $\mu_f$ is a Rajchman measure and 
$\mu_{f,z}$ with 
\[
d\mu_{f,z} := \gamma(z)\, d\overline{\mu_f}
\]
 is absolutely continuous with respect to $\overline{\mu_f}$, \cite[Lemma 4, p.\ 77]{DescriptiveSetTheory} yields that $\mu_{f,z}$ is a Rajchman measure as well. Thus,  we have
\begin{equation*}
T^n f(z) = \int \zeta^{n} \gamma(z)(\zeta) d\overline{\mu_f}(\zeta) = \int \zeta^{n} d \mu_{f,z}(\zeta) = \hat{\mu}_{f,z}(-n) \to 0 
\end{equation*}
and 
\begin{equation*}
S_n \mu_f(z) = \int \zeta^{-n} \gamma(z)(\zeta) d\overline{\mu_f}(\zeta) = \int \zeta^{-n} d \mu_{f,z}(\zeta) = \hat{\mu}_{f,z}(n) \to 0
\end{equation*}
as $n$ tends to $\infty$. Furthermore, for all $ n \in \mathbb{N}$ we have 
\begin{equation*}
\vert T^n f (z) \vert \leq \int \vert \gamma(\zeta) \vert\, d\vert \mu_f \vert (\zeta) \quad \text{and} \quad \vert S_n \mu_f (z) \vert \leq \int \vert \gamma(\zeta)\vert\,  d\vert \mu_f \vert (\zeta) 
\end{equation*}
where $\int_{B} \vert \gamma(\zeta)\vert\, d\vert \mu_f \vert (\zeta)$ is $p$-integrable on $\Omega$ by assumption. Lebesgue's theorem of dominated convergence yields that $\Vert T^n f \Vert_p \to 0$  and $\Vert S_n \mu_f \Vert_p \to 0$ as $n$ tends to $\infty$.
\end{proof}

%We now want to use this type of functions in order to approximate functions in the Bergman spaces for $p \geq 2$.

As noted in the introduction, for $p \ge 2$ and  $\zeta \in \partial \om^*$ the functions $\gamma(\zeta)$ are in general not $p$-integrable.  We introduce appropriate  means of the $\gamma(\zeta)$ which turn out to be integrable for all $p$. 

\begin{remark} \label{remarkApproxFunctions}
It is easily seen (see e.g. \cite[Theorem 1.7]{Hedenmalm}) that 
\begin{equation*}
\int_\mathbb{T} \frac{dm(\alpha)}{\vert 1 - \alpha z \vert} = O \left( \log \frac{1}{1 - \vert z \vert}\right) \quad (\vert z \vert \to 1^-).
\end{equation*}
Since, for all $p \geq 1$,
\begin{equation*}
\int_\mathbb{D}\left| \log\frac{1}{1-\vert z \vert} \right| ^p dm_2(z) \leq \frac{1}{\pi}\int_0^1 \left| \log\frac{1}{1-r} \right|^p dr<\infty,
\end{equation*} 
by symmetry we obtain that 
\begin{equation*}
\int \vert \gamma (\alpha)\vert d m(\alpha) \in L^p(\C_\infty \setminus \T,m_2).
\end{equation*}
For a Borel set $B \subset \mathbb{T}$ we define  $dm_B = 1_B dm$ and 
\begin{equation} \label{approxFunctions}
f_{B} \coloneqq Cm_B = \int \gamma(\alpha)\, dm_B(\alpha) = \int_{B} \frac{d m(\alpha)}{1 - \alpha \, \cdot} \in H(\overline{B}^\ast).
\end{equation} 
Let now $\Omega \subset \C_\infty$ be an open set and $\Omega^* \cap \mathbb{T} \neq \emptyset$. Then, for all Borel sets $B \subset\Omega^* \cap \mathbb{T}$ and for all $1 \leq p < \infty$
\begin{equation*}
\int \vert \gamma (\alpha)\vert\, d m_B(\alpha) \in L^p(\Omega,m_2),
\end{equation*}
 which yields that $m_B$ is a measure in $\mathcal{M}_p(\om)$ supported on $B$ and hence $f_B \in \mathcal{C}_p(\om)$. Since $1_B \in L^1(\mathbb{T})$ and the arc length measure is a Rajchman measure, Theorem \cite[Lemma 4, p. 77]{DescriptiveSetTheory} yields that $m_B$ is a Rajchman measure as well. 
\end{remark}

The following result shows that under the conditions of Theorem \ref{mixingTaylorShift} the functions $f_B$ densely span $A^p(\om)$. 
%For the proof we refer to Section  \ref{proofs}.
 
\begin{thm} \label{approxForCompactSetsp>=2}
Let $1 \leq p < \infty$ and $0 \in \Omega\subset \Cinf$ be a domain which is either bounded in $\C$ or contains $\infty$. If each component $K$ of $\Omega^*$ is the closure of a simply connected domain containing a rectifiable Jordan curve $\Gamma$ such that the linear measure of $\Gamma \cap \T$ is positive, then the
span of $\{f_{B}: B \subset \Omega^* \cap \mathbb{T}\}$ is dense in $A^p(\Omega)$.  
\end{thm}
\begin{proof}
We first assume that $\Omega$ is bounded in $\C$ and fix $\ell \in A^p(\Omega)'$ with $\ell(f_{B}) = 0$ for all Borel sets $B \subset \Omega^* \cap \mathbb{T}$. Again, according to the Hahn-Banach theorem there exists a function $g \in L^q(\Omega)$, where $q$ is the conjugated exponent,  such that 
\begin{equation*}
\ell (f) =  \int_{\Omega} f\overline{g} \, dm_2 
\end{equation*}
for all $f \in A^p(\Omega)$, and for $1_\Omega g d m_2\in M(\overline{\Omega})$ the Cauchy transform 
\begin{equation*}
(Vg)(\alpha) = \int_{\Omega} \frac{\overline{g}(\zeta)}{1 - \zeta \alpha } dm_2(\zeta)
\end{equation*}
of $1_\Omega g d m_2$ is holomorphic in the interior of $\om^\ast$. However, since  $q\le 2$ for $p\ge 2$, it is no longer guaranteed that the Cauchy integral is defined and continuous on $\C$.

 Since $\int_{\om^\ast \cap\T} \vert \gamma( \alpha) \vert\, dm( \alpha) \in L^p(\Omega)$, Hölder's inequality yields
\[
\int_{\Omega}  \int_{\om^\ast \cap\T} \vert \frac{g(\zeta)}{1 - \zeta \alpha } \vert \, d m( \alpha) \, d m_2(\zeta) \leq \|g\|_q \cdot \|\int_{\om^\ast \cap\T} \vert \gamma( \alpha) \vert\, dm( \alpha) \|_p < \infty.
\]
Hence the maximal Cauchy transform  
\[
\int_{\Omega} \frac{|g(\zeta)|}{|1 - \zeta \alpha| } dm_2(\zeta)
\]
is finite for $m$-almost all on $\alpha \in \om^\ast \cap \T$ and $Vg$ exists $m$-almost everywhere on $\om^\ast \cap \T$. Moreover, for all Borel sets $B \subset \Omega^* \cap \mathbb{T}$ we may apply Fubini's theorem to get
\begin{equation*}
0=\ell(f_{B}) = \int_{\Omega} \int_B \frac{dm ( \alpha)}{1 - \zeta  \alpha}\, \overline{g}(\zeta)\, dm_2(\zeta) = \int_B Vg ( \alpha)\, dm( \alpha).
\end{equation*}
This implies that  $Vg = 0$ $m$-almost everywhere on $\Omega^* \cap \mathbb{T}$.

 %In particular, $C\mu $ vanishes on a subset of $\Omega^*\cap \mathbb{T}$ of positive linear measure. %Because of the continuity of $C\mu$ on $\C$, we obtain $C\mu \vert_{\Omega^* \cap \mathbb{T}} \equiv 0$.

Let $G$ be a bounded simply connected domain in $\C$ and let $D_q(G)$ denote the the Dirchlet space of order $q$  defined as in Remark \ref{suff}.
Fixing a point $\beta \in G$, we equip $D_q(G)$ with the (complete) norm
\[
\| h\|_{q}=|h(\beta)|+\Big(\int_G |h'|^q\, d\lambda_2\Big)^{1/q}.
\]
If $\vp$ is the conformal mapping from $\D$ to $G$ with $\vp(0)=\beta$ and $\vp'(0)>0$ then  
\[
h \mapsto (h \circ \vp)(\vp')^{2-q}
\] 
defines an isomorphism between $D_q(G)$ and the Dirichlet space $D_q:=D_q(\D)$ on the unit disc. It is known that $D_q \subset H^q$, where $H^q$ denotes the Hardy space of order $q$ (see e.g. \cite[p.\ 88]{CimaRoss}). In particular, for $h \in D_q(G)\subset D_1(G)$ we have $(h\circ \vp)\vp' \in H^1$, which in turn implies that $h$ belongs to the Hardy-Smirnov space $E^1(G)$ (see \cite[Corollary to Theorem 10.1]{Duren}).

Let now $K$ be a component of $\om^*$ and $G$ the interior of $\Gamma$. Then  the harmonic measure $\omega(\cdot, K \cap \T,G)$ is positive or $G$ meets $\T$. In a similar way as in the proof of Theorem 1 in \cite{NY}, by applying Theorem 3, Chapter II, Section 4,  from \cite{Ste} it can be shown that $Vg|_G$ belongs to $ D_q(G)$. Since $Vg=0$ $m$-almost everywhere on $\om^\ast \cap \T$, in the case of positive  harmonic measure $\omega(\cdot, K \cap \T,G)$ the local F.\ and M.\ Riesz theorem implies that $Vg$ vanishes on a subset of $K \cap \T$ of positive harmonic measure (cf. Remark \ref{suff}). 

Since $\Gamma:=\partial G$ is a rectifiable Jordan curve, the set of cone points of $\Gamma$ has full linear measure (see \cite[Corollary 1.3 and p.\ 207]{Garnett} or \cite[Section 3.5]{Duren}). Hence $m$-almost every point in $K \cap \T$ is a cone point. it With similar arguments as in the proof of Theorem 3.2.4 in \cite{FKMR} it can be shown  that  the non-tangential limit of $Vg$ at $\alpha$ coincides with $(Vg)(\alpha)$ and hence with $0$ at $m$-almost all $\alpha \in K \cap \T$. From \cite[Theorem 10.3]{Duren} we obtain that $Vg|_G = 0$. As $K$ was an arbitrary component of $\om^*$, it follows that $Vg|_{(\om^*)^\circ}=0$ and then
 \[
 \ell (\gamma(\alpha))=(Vg)(\alpha)= 0
 \]
  for all $\alpha  \in (\om^*)^\circ$. Since, by assumption, $\Omega$ is a domain, it follows that the inner boundary of $\overline{\Omega}$ is empty, which allows to apply  \cite[Corollary  p.\ 162]{Hedberg} showing that the rational functions with (simple) poles in $\C \setminus \overline{\om}$ are dense in $A^p(\om)$. This implies that $\ell=0$ and thus the denseness of $\linspan\{f_{B}: B \subset \Omega^* \cap \mathbb{T}\}$  in $A^p(\Omega)$. 

Along the same lines, we get in case of $\om$ containing $\infty$ and $\om_\rho:=\om \cap \rho\D$ that the span of $\{f_{B}: B \subset \Omega^* \cap \mathbb{T}\} \cup \gamma (\rho^{-1}\D)$
is dense in $A^p(\Omega_\rho)$. According to Remark \ref{decomp}, the span of $\{f_{B}: B \subset \Omega^* \cap \mathbb{T}\}$ is dense in $A^p(\Omega)$.
\end{proof}

\begin{proof}[Proof of Theorem \ref{mixingTaylorShift}]
By Theorem \ref{approxForCompactSetsp>=2}, the span $L$ of  
$ \{f_{B}: B \subset \Omega^* \cap \mathbb{T}\}$
is dense in $A^p(\Omega)$. Furthermore, for $f=\sum_B\lambda_B f_B \in L$ let $(S_n \mu_f)_{n \in \mathbb{N}}$ be the sequence  defined in (\ref{definitionS_n}) with $\mu_f=\sum_B \lambda_B m_B$ . Then Lemma \ref{iteratesConvergeTo0} yields that $\Vert T^n f \Vert_p$ and $\Vert S_n \mu_f \Vert_p$ converge to 0 as $n\to \infty$. Hence, applying Lemma \ref{iteratesT^nr_nuOnA^p} and interchanging integration and $T^n$ we obtain  for $n \in \mathbb{N}$ 
\begin{equation*}
T^nS_n \mu_f = T^n \left( \int \gamma(\zeta)\zeta^{-n}\, dm(\zeta)\right) = \int T^n\gamma(\zeta)\zeta^{-n}\, dm(\zeta) = f.
\end{equation*}
The Kitai criterion in the version \cite[Exercise 3.1.1 or Remark 3.13]{LinChaos} yields the assertion.

Finally, the denseness of the span of  $\{f_{B}: B \subset \Omega^* \cap \mathbb{T}\}$ implies that in the case $\D \subset \om$ and $p\ge 2$ the space $A^p(\om)$ is (continuously and) densely embedded in $A^2(\D)$. Since the Taylor shift is not frequently hypercyclic on $A^2(\D)$ (\cite[Example 9.18]{LinChaos}) it is also not frequently hypercyclic on $A^p(\om)$. 
\end{proof}

%\section{Mean Approximation} \label{proofs}

\bibliographystyle{apalike}
\bibliography{disbib}{}\vspace*{2em}
Address: University of Trier, FB IV, Mathematics,  D-54286 Trier, Germany\\
e-mail: jmueller@uni-trier.de; maikethelen@web.de

\end{document}